\documentclass[a4paper,reqno]{amsart}
\usepackage{amssymb, amsmath, amscd}

\def\pbg{\vphantom{\vrule height 11pt}}
\date{13 October 2014 - version 7x}

\def\scalar{K}
\newtheorem{theorem}{Theorem}[section]
\newtheorem{lemma}[theorem]{Lemma}
\newtheorem{remark}[theorem]{Remark}
\newtheorem{definition}[theorem]{Definition}
\newtheorem{example}[theorem]{Example}
\makeatletter
\renewcommand{\theequation}{%
\thesection.\alph{equation}} \@addtoreset{equation}{section}
\makeatother
\begin{document}
\makeatletter
\renewcommand{\theequation}{%
\thesection.\alph{equation}} \@addtoreset{equation}{section}
\makeatother
\title[Proper surfaces]
{Real analytic complete non-compact surfaces in Euclidean space with finite total curvature arising
as solutions to ODEs}
\author
{P. Gilkey, C.Y. Kim, and J. H. Park${}^1$}
\address{PG: Mathematics Department, University of Oregon,
Eugene OR 97403 USA}
\email{gilkey@uoregon.edu}
\address{CYK \& JHP:Department of Mathematics,
Sungkyunkwan University,
Suwon, 440-746, Korea}
\email{intcomplex@skku.edu, parkj@skku.edu}
\begin{abstract} We use the solution space
of a pair of ODEs of
at least second order
 to construct a smooth surface in Euclidean space.
We describe when this surface is a proper embedding which is geodesically
complete with finite total Gauss curvature. If the associated roots of the ODEs
are real and distinct, we give a universal upper bound for
the total Gauss curvature of the surface which depends only on the orders of the ODEs
and we show that the total Gauss curvature of the surface vanishes if the ODEs are second
order. We examine when the surfaces are asymptotically minimal.
\end{abstract}
\subjclass[2010]{53A05 (primary); 53C21 (secondary)}
\keywords{geodesically complete surface, finite total Gauss curvature, Gauss-Bonnet theorem, asymptotically minimal,
constant coefficient ordinary differential equation.\\${}^1$ Corresponding author}
\maketitle

\section{Introduction}

\subsection{Historical context}
Let $\Sigma$ be a finitely connected non-compact geodesically
complete Riemann surface. If the Gauss curvature $K$
 is integrable with respect to the Riemannian
element of volume, $\operatorname{dvol}$, then the {\it total Gauss
curvature} is given by
$\scalar[\Sigma]:=\int_\Sigma K\operatorname{dvol}$.
The total Gauss curvature plays an important role in many settings --
and the role is subtly different in each application.
Cohn-Vossen \cite{CV35, CV36} showed that
\begin{equation}\label{E1.a}
\scalar[\Sigma]\le 2\pi\chi(\Sigma)\,.
\end{equation}
Subsequently, Huber \cite{H57} reproved this result and showed
additionally that if the total volume of $\Sigma$ was finite, then equality
holds. We also refer to a more recent derivation of Equation~(\ref{E1.a})
by Bleecker \cite{B74} of using work of Chern.
Higher dimensional analogues have been studied --
see, for example, Dillen and K\"uhnel \cite{DK05}.

Mafra \cite{M07} examined the question of whether
a holomorphic curve in $\mathbb{C}^2$ with finite total Gauss curvature
is contained in an algebraic curve. Shioya \cite{S94}
showed that if $\scalar[\Sigma]<2\pi$, then any maximal
geodesic outside a sufficiently large compact set in $\Sigma$ forms almost the same shape as that of a maximal geodesic in a flat cone.
Shioya \cite{S04} subsequently considered the case where
$\scalar[\Sigma]=2\pi$ (see also related work in
Shiohama et al. \cite{S85,SST03, SST90}).
Carron et al. \cite{CEK04} showed the existence of geometrically
bound states if $\scalar[\Sigma]<\infty$ and $\Sigma$ is not homeomorphic
to the plane. Li et al. \cite{LLT13} examined conformal maps
of the 2-disk into $\mathbb{R}^n$ under the condition that the total
Gauss curvature was at most $2\pi$.

The total Gauss curvature
is central to the study of minimal surfaces.
If the surface is minimal,
Chern and Osserman \cite{CO67} improved Equation~(\ref{E1.a})
to become
$$\scalar[\Sigma]\le 2\pi(\chi(\Sigma)-e)$$
where $e$ denotes the number of ends.
We refer to subsequent work of
Jorge and Meeks \cite{JM83}, and Kobuku et al. \cite{KUY02} (among
others).
We also refer to the discussion in Chen and Cheng \cite{CC00} or
Seo \cite{S10} where the ambient space is
$\mathbb{H}^n$, to Esteve and Palmer \cite{EP14}
where the ambient manifold is a Cartan-Hadamard manifold,
and to Ma \cite{M13}  and Ma, Wang, and Wang \cite{MWW13}
 where the ambient space is Lorentzian.

Integrals of the Gauss curvature are not only important in the 2-dimensional
setting. For example, Willerton \cite{W14} used the total Gauss curvature
to examine the leading terms in the magnitude of
an arbitrary homogeneous Riemannian manifold. Hwang et al. \cite{HCY13}
used the total Gauss curvature to study the eigenvalues of the Laplacian.
The total Gauss curvature plays an important role in Ricci flow. Li \cite{L13}
showed the lowest eigenvalue in a family of geometric
operators was monotonic under the normalized Ricci flow if the initial
manifold had nonpositive total Gauss curvature. Chow et al. \cite{CLY12}
gave a necessary and sufficient condition for the asymptotic volume
ratio to be positive that involved the average Gauss curvature.
That total Gauss curvature has also been studied for connections other
than the Levi-Civita connection, see, for example,
the discussion in Stephanov et al. \cite{STM11}.

\subsection{Outline of the paper}
In this paper, we shall discuss a family of non-compact
real analytic isometric embeddings $\Sigma$ of the plane
in Euclidean space $\mathbb{R}^n$ which arise as the solution space to a pair of ODE's.
The condition that $\Sigma$
is real analytic is, of course, important as otherwise one could simple take a flat plane and put
a small bump in it; this would, of course produce $\scalar[\Sigma]=0$ and
for many of our examples, $\scalar[\Sigma]$ is strictly negative.

We shall assume that all the roots of the associated characteristic
polynomials are simple to avoid notational complexities with the
multiplicities; the second author is investigating what
happens when the roots have higher multiplicities in his thesis. We
shall also assume that the real roots of the associated
characteristic polynomials are dominant, i.e. control the asymptotic
behavior of the embedding at infinity. Under these conditions, we
will show in Theorem~\ref{T1.7} that the surface $\Sigma$ is
properly embedded, is geodesically complete, and has infinite
volume. We will also show in Theorem~\ref{T1.8} that the Gauss
curvature $K\in L^1(\Sigma,\operatorname{dvol})$ and hence the total
Gauss curvature $\scalar[\Sigma]$ is well defined. In
Example~\ref{Ex8.2}, we show that $|\scalar|[\Sigma]$ can be
infinite if the real roots are not dominant.

In Theorem~\ref{T1.11}, we use the Gauss-Bonnet Theorem to express $K[\Sigma]$
in terms of integrals along the coordinate curves.
The case where the two ODE's are second order is particularly
tractable; we will use Theorem~\ref{T1.11} to prove Theorem~\ref{T1.12}
which shows that $\scalar[\Sigma]=0$ if $n_1=n_2=2$.
In Example~\ref{Ex8.5}, we show $\scalar[\Sigma]$ can be negative if $n_1=n_2=3$ so
this result is non-trivial. If all the roots of the associated ODE's are real, we will
show in Theorem~\ref{T1.13} that there is a uniform
upper bound for $|K|[\Sigma]$ which depends only on the dimension; again, this uses Theorem~\ref{T1.11}.
In Example~\ref{Ex8.6}, we will provide a family of examples $\Sigma_k$
where this condition fails and where $\lim_{k\rightarrow\infty}K[\Sigma_k]=-\infty$.
If all the roots are real and if there are at least two positive and at least two negative roots for each ODE,
we show in Theorem~\ref{T1.15} that the mean curvature vector $H$ goes to zero at infinity and that
$H\in L^3(\Sigma,\operatorname{dvol})$ so $\Sigma$ is asymptotically minimal; in Example~\ref{Ex8.7}
we show the condition that there are at least two roots of each sign is essential in this regard and in
Example~\ref{Ex8.8} we show that $p=3$ is optimal if a uniform estimate is required.

The present paper grew out of the study of curves of finite total
first curvature given by an ODE with two other authors \cite{GKMPY14}. We begin
by reviewing these results for the convenience of the reader
as many of our subsequent theorems depend on these results.
The rest of the introduction is then a careful statement of the
main results of the paper. Section~\ref{S2} is an introduction
to the geometry of surfaces embedded in $\mathbb{R}^n$
and expresses the relevant geometric quantities we shall need
in terms of the exterior algebra as this is a convenient formalism
for our purposes. In Section~\ref{Sect-3}, we demonstrate
Theorem~\ref{T1.7}, in Section~\ref{S4} we establish
Theorem~\ref{T1.8},  in Section~\ref{S5} we derive Theorem~\ref{T1.11},
and in Section~\ref{S6}
we use the Gauss-Bonnet theorem to prove Theorem~\ref{T1.12};
this express $\scalar[\Sigma]$ in terms
of the curves defined by the two ODE's and plays a central
role in the proof of Theorem~\ref{T1.13} which gives a uniform estimate
for $\scalar[\Sigma]$. In Section~\ref{S7}, we examine the norm of the mean curvature
vector and prove Theorem~\ref{T1.15}. We conclude the paper in Section~\ref{S8} by
presenting some Mathematica calculations using a Mathematica program constructed
by M. Brozos-Vazquez \cite{M2} to discuss various illustrative examples.

\subsection{Curvature}
If $\vec u,\vec v\in\mathbb{R}^n$, let $(\vec u,\vec v):=u^1v^1+\dots u^nv^n$ and $||\vec u||^2:=(\vec u,\vec u)$.
We extend $(\cdot,\cdot)$ to an inner product on tensors on all types and, in particular, to the exterior algebra on
$\mathbb{R}^n$.
If $\sigma:\mathbb{R}\rightarrow\mathbb{R}^n$ is an immersed curve, then
the {\it element of arc length} $ds=ds_\sigma$,
the {\it first curvature} $\kappa=\kappa_\sigma$, and the {\it total first curvature} $\kappa[\sigma]$ are defined, respectively, by:
\begin{eqnarray*}
&&ds:=||\dot\sigma(t)||dt,\quad
\kappa_\sigma(t):=\frac{||\dot\sigma(t)\wedge\ddot\sigma(t)||}{||\dot\sigma(t)||^3},\\
&&\kappa[\sigma]:=\int_\sigma\kappa ds
=\int_{-\infty}^\infty\frac{||\dot\sigma(t)\wedge\ddot\sigma(t)||}{||\dot\sigma(t)||^2}dt\,.
\end{eqnarray*}
If $\Sigma$ is an immersed surface in $\mathbb{R}^n$, let $\operatorname{dvol}$
be the Riemannian measure and let
$\scalar$ be the Gauss curvature.
If $|\scalar|$ is in $L^1(M,\operatorname{dvol})$,
let $\scalar[\Sigma]:=\int_\Sigma\scalar\operatorname{dvol}$.

\subsection{Curves defined by ODEs}
We review briefly some previous results that we shall need and refer
to the discussion in \cite{GKMPY14} for further details.
If $\phi=\phi(t)$ is a smooth real valued function,
let $\phi^{(i)}$ be the $i^{\operatorname{th}}$ derivative. Let
$$P(\phi):=\phi^{(n)}+c_{n-1}\phi^{(n-1)}+\dots+c_0\phi$$
be
 a real constant coefficient ordinary differential operator of order $n\ge2$.
Let $\mathcal{S}=\mathcal{S}(P)$ be the solution space of $P$,
let $\mathcal{P}=\mathcal{P}(P)$ be the characteristic polynomial of $P$, and
let $\mathcal{R}=\mathcal{R}(P)$ be the roots of $\mathcal{P}$:
$$\begin{array}{ll}
\mathcal{S}:=\{\phi\in C^\infty(\mathbb{R}):P(\phi)=0\},&
\mathcal{P}(\lambda):=\lambda^n+c_{n-1}\lambda^{n-1}+\dots+c_0,\\
\mathcal{R}:=\{\lambda\in\mathbb{C}:\mathcal{P}(\lambda)=0\}\,.
\pbg
\end{array}$$
Enumerate the roots in the form
 $\mathcal{R}=\{s_1,\dots,s_k,z_1,\dots,z_u,\bar z_1,\dots,\bar z_u\}$ where the $\{s_i\}$ are
the distinct real roots of $\mathcal{P}$ for $1\le i\le k$ and where
the $\{z_j=a_j+b_j\sqrt{-1}\}$ are the distinct complex roots of $\mathcal{P}$ for $1\le j\le u$
where $b_j>0$. We order the real roots so $s_1>\dots>s_k$ and the complex roots
so $a_1\ge a_2\ge\dots$. If there are no real roots, we set $k=0$;
if there are no complex roots, we set $u=0$.

\begin{definition}
\rm Let $\Re(\cdot)$ be the real part of a complex number.
 We shall say that a root $\lambda\in\mathcal{R}$ is {\it dominant} if $\Re(\lambda)>\Re(\mu)$
 for all $\mu\in\mathcal{R}-\{\lambda\}$
and if $\Re(\lambda)>0$ or if $\Re(\lambda)<\Re(\mu)$ for
all $\mu\in\mathcal{R}-\{\lambda\}$ and if $\Re(\lambda)<0$;
note that a dominant root is necessarily real and is either $s_1$ or $s_k$.
One has that $s_1$ is dominant if $s_1>0$ and if $s_1>a_1$ and similarly that $s_k$ is dominant if $0>s_k$
and if $a_u>s_k$.
If $s_1$ is dominant and if $\lambda\in\mathcal{R}-\{s_1\}$ satisfies $\Re(\lambda)\ge\Re(\mu)$ for
all $\mu\in\mathcal{R}-\{s_1\}$, then we say $\lambda$ is {\it sub-dominant}. Similarly if $s_k$ is dominant
and if $\lambda\in\mathcal{R}-\{s_k\}$ satisfies $\Re(\lambda)\le\Re(\mu)$ for all $\mu\in\mathcal{R}-\{s_k\}$,
then we say $\lambda$ is {\it sub-dominant}.
\end{definition}

If all the roots are simple (i.e. have multiplicity 1 so $n=k+2u$), then the canonical
basis for the solution space $\mathcal{S}$ consists of the functions:
\begin{equation}\label{E1.b}
\{e^{s_it},\ e^{a_jt}\cos(b_jt),\ e^{a_jt}\sin(b_jt)\}\text{ for }1\le i\le k\text{ and }1\le j\le u\,.
\end{equation}
The functions $\{e^{s_it}\}$ do not appear, of course, if there are no real roots and, similarly, the functions
$\{e^{a_jt}\cos(b_jt),e^{a_jt}\sin(b_jt)\}$ do not appear if there are no complex roots.
More generally, if $s_i$ is a real root of multiplicity $\nu\ge2$, then we must replace the single
function $e^{s_it}$ in Equation~(\ref{E1.b}) by the $\nu$ functions
$$
\{e^{s_it},\ te^{s_it},\dots,\ t^{\nu-1}e^{s_it}\}
$$
while if $z_j$ is a complex root of multiplicity $\nu\ge2$, then we must replace
the pair of functions $\{e^{a_jt}\cos(b_jt),e^{a_jt}\cos(b_jt)\}$ in Equation~(\ref{E1.b}) by the $2\nu$
functions:
$$\begin{array}{l}
\{e^{a_jt}\cos(b_jt),\
te^{a_jt}\cos(b_jt),\dots,
t^{\nu-1}e^{a_jt}\cos(b_jt),\\
\phantom{\{}e^{a_jt}\sin(b_jt),\
te^{a_jt}\sin(b_jt),\dots,
t^{\nu-1}e^{a_jt}\sin(b_jt)\}\,.\pbg
\end{array}$$
Let $\{\phi_1,\dots,\phi_n\}$ be an enumeration of the canonical basis for $\mathcal{S}$
described above.
We define the {\it associated curve} $\sigma=\sigma_P$ by setting:
$$
\sigma(t):=(\phi_1(t),\dots,\phi_n(t)):\mathbb{R}\rightarrow\mathbb{R}^n\,.
$$
If $\{s_1,s_k\}$ are dominant roots, then
these roots control the behavior of $\|\sigma\|$ at infinity, i.e.:
$$
\lim_{t\rightarrow\infty}e^{-s_1t}\|\sigma\|=1,\quad\text{ and }\quad
\lim_{t\rightarrow-\infty}e^{-s_kt}\|\sigma\|=1\,.
$$
We refer to \cite{GKMPY14} for the proof of the following result:

\begin{theorem}\label{T1.2}
If all the roots of $\mathcal{P}$ are simple and if $\{s_1,s_k\}$ are dominant roots, then
$\sigma$ is a proper embedding
of $\mathbb{R}$ in $\mathbb{R}^n$ of infinite length with $\kappa[\sigma]<\infty$.
\end{theorem}

\begin{remark}\rm In Example~\ref{Ex8.1} we will see that $\kappa[\sigma]$
can be infinite if there exists a complex root $\lambda\in\mathcal{R}$ with $\Re(\lambda)$ maximal or
minimal.\end{remark}

We have taken the standard inner product on $\mathbb{R}^n$
to define the element of arc length $ds$ and the geodesic curvature $\kappa$.
The precise inner product is irrelevant; Theorem~\ref{T1.2}
continues to hold for an arbitrary positive definite inner product on $\mathbb{R}^n$.
Equivalently, this shows that it is not necessary to choose the standard basis for $\mathcal{S}$
in defining $\sigma$; any basis will do. Consequently, Theorem~\ref{T1.2} is really a result about the
solution space $\mathcal{S}$. If $a_1\ge s_1$ or $s_k\ge a_u$, then
the dominant exponential involves $\sin$ and $\cos$. This implies that the total first curvature is
infinite. There are analogous results when multiple roots are permitted; as
they are a bit more complicated to state, we shall refer to \cite{GKMPY14} for details.

If all the roots of $P$ are real and simple, then the associated curve is of the form
$$
\sigma(t)=\sigma_{s_1,\dots,s_n}(t):=(e^{s_1t},\dots,e^{s_nt})\text{ where }s_1>\dots>s_n\,.
$$
There is a uniform estimate for the total first curvature \cite{GKMPY14} of such a curve:
\begin{theorem}\label{T1.4}
  $\kappa[\sigma_{s_1,\dots,s_n}]\le2(n-1)n$.
\end{theorem}

\begin{remark}\rm Let $\sigma_k(t):=(e^t,\cos(kt),\sin(kt),e^{-t})$ so
$\mathcal{R}_k=\{\pm1,\pm\sqrt{-1}\}$.
In Example~\ref{Ex8.4}, we will show that $\lim_{k\rightarrow\infty}\kappa[\sigma_k]=\infty$ so the
assumption all the roots are real is essential to establish a uniform upper bound.
\end{remark}

\subsection{Surfaces defined by a pair of ODEs}\label{S1.5}
We establish some basic notational conventions for the remainder of the paper.
Let $\mathcal{R}_1$ (resp. $\mathcal{R}_2$) be the roots and let $\sigma_1$ (resp. $\sigma_2$) be the curve defined by the ODE $P_1$ (resp. $P_2$).
We assume that all the roots are simple and express:
\begin{eqnarray*}
&&\mathcal{R}_1=\{r_1,\dots,r_k,a_{1}\pm b_1\sqrt{-1},\dots a_p\pm b_p\sqrt{-1}\},\\
&&\mathcal{R}_2=\{s_1,\dots,s_\ell,c_1\pm d_1\sqrt{-1},\dots, c_q\pm d_q\sqrt{-1}\},\\
&&\sigma_1(t_1):=(e^{r_1t_1},\dots,e^{r_kt_1},e^{a_1t_1}\cos(b_1t_1),e^{a_1t_1}\sin(b_1t_1),\dots),\\
&&\sigma_2(t_2):=(e^{s_1t_2},\dots,e^{s_\ell t_2},e^{c_1t_2}\cos(d_1t_2),e^{c_1t_2}\sin(d_1t_2),\dots)\,.
\end{eqnarray*}
Let $n=n_1n_2$ and let $\Sigma:\mathbb{R}^2\rightarrow\mathbb{R}^n$ be defined by:
$$\Sigma(t_1,t_2):=\sigma_1(t_1)\otimes\sigma_2(t_2)\,.$$
If $\{\phi_{1,1},\dots,\phi_{n_1,1}\}$ (resp. $\{\phi_{1,2},\dots,\phi_{n_2,2}\}$) is the standard basis
for the solution space of $P_1$ (resp. $P_2$), then the coordinates of $\Sigma$ are the collection of functions
$\{\phi_{i,1}(t_1)\phi_{j,2}(t_2)\}$ for $1\le i\le n_1$ and $1\le j\le n_2$.
\begin{definition}\rm
Let $P_1$ and $P_2$ be real ODEs with simple roots. We say that {\it the real roots are dominant}
if $\{r_1,r_k\}$ are dominant roots for $P_1$ and $\{s_1,s_\ell\}$ are dominant roots for $P_2$.

\end{definition}
We shall establish the following generalization of Theorem~\ref{T1.2} in Section~\ref{S2}.
\begin{theorem}\label{T1.7}
If all the roots of $\mathcal{P}_1$ and of $\mathcal{P}_2$ are simple and if the real roots are dominant roots, then
$\Sigma$ is a proper
embedding of $\mathbb{R}^2$ in $\mathbb{R}^{n}$ which is geodesically complete and
which has infinite volume.
\end{theorem}

We shall establish the following result in Section~\ref{S4}:

\begin{theorem}\label{T1.8}
If all the roots of $\mathcal{P}_1$ and of $\mathcal{P}_2$ are simple and if the real roots are dominant,
then there exist $\epsilon=\epsilon(\Sigma)>0$ and
$C=C(\Sigma)>0$ so
\begin{enumerate}
\item $|\scalar(t_1,t_2)|\le C e^{-\epsilon\|(t_1,t_2)\|}$.
\item $g|\scalar(t_1,t_2)|\le Ce^{-\epsilon\|(t_1,t_2)\|}$.
\item $|\scalar|[\Sigma]<\infty$.
\end{enumerate}
\end{theorem}

\begin{remark}\rm We will show in Example~\ref{Ex8.2} that this can fail if the
real roots are not dominant.
\end{remark}

\begin{definition}\label{D1.10}
\rm Let $\sigma$ be an immersed curve in $\mathbb{R}^n-\{0\}$ so that $\sigma\wedge\dot\sigma\ne0$; this is the case if
$\sigma$
is defined by a constant coefficient ODE of course. We define:
$$
\Theta_\sigma(t):=\frac{(\dot\sigma(t)\wedge\sigma(t),\dot\sigma(t)\wedge\ddot\sigma(t))}
{\|\dot\sigma(t)\wedge\sigma(t)\|\cdot\|\dot\sigma(t)\|^3}\,.
$$
If $\Theta ds$ is integrable, we set
$$\Theta[\sigma]:=\int_{-\infty}^\infty\Theta(\sigma)ds=\int_{-\infty}^\infty
\frac{(\dot\sigma(t)\wedge\sigma(t),\dot\sigma(t)\wedge\ddot\sigma(t))}
{\|\dot\sigma(t)\wedge\sigma(t)\|\cdot\|\dot\sigma(t)\|^2}dt\,.
$$
\end{definition}
We use the Cauchy-Schwarz inequality to see
\begin{equation}\label{E1.c}
|\Theta_\sigma(t)|\le\kappa_\sigma(t)\quad\text{ so }\quad
|\Theta|[\sigma]\le\kappa[\sigma]\,.
\end{equation}
Consequently, if all the roots of $\mathcal{P}$ are simple and if the real roots are dominant, then
$\Theta[\sigma]:=\int_\sigma\Theta ds$ is well defined. For example, if $\sigma(t)=e^{at}e_1+e^{bt}e_2$
is a curve in $\mathbb{R}^2$ for $a>0>b$,
then
\begin{eqnarray}
&&\Theta(\sigma)ds=\frac{e^{(2a+2b)t}\left((ae_1+be_2)\wedge(e_1+e_2),
(ae_1+be_2)\wedge(a^2e_1+b^2e_2)\right)}
{e^{(a+b)t}\|(ae_1+be_2)\wedge(e_1+e_2)\|\cdot\{a^2e^{2at}+b^2e^{2bt}\}}dt\nonumber\\
&&\phantom{\Theta(\sigma)ds}=e^{(a+b)t}
\frac{\left((a-b)e_1\wedge e_2,(ab^2-a^2b)e_1\wedge e_2\right)}
{\|(a-b)e_1\wedge e_2\|\cdot\{a^2e^{2at}+b^2e^{2bt}\}}dt\label{E1.d}\\
&&\phantom{\Theta(\sigma)ds}=\frac{(a-b)ab(b-a)}{|a-b|}
\frac{e^{(a+b)t}}{a^2e^{2at}+b^2e^{2bt}}dt\,.\nonumber
\end{eqnarray}
Since $a>0>b$, the coefficient is
$|(a-b)ab|>0$ and $\Theta[\sigma]=\kappa[\sigma]>0$.

If $P_1$ and $P_2$ are admissible, then $|\scalar|[\Sigma]$ is finite and we set
$\scalar[\Sigma]:=\int_\Sigma\scalar\operatorname{dvol}$.
We will use the Gauss-Bonnet theorem to establish the following result
in Section~\ref{S5}:

\begin{theorem}\label{T1.11}
If all the roots of $\mathcal{P}_1$ and of $\mathcal{P}_2$ are simple and
if the real roots are dominant, then
$0=\scalar[\Sigma]-2\Theta[\sigma_1]-2\Theta[\sigma_2]+2\pi$.
\end{theorem}

The 4-dimensional setting is particularly tractable. We will establish the following result in Section~\ref{S6}:

\begin{theorem}\label{T1.12}
If all the roots of $\mathcal{P}_1$ and of $\mathcal{P}_2$ are simple, if the real roots are dominant,  and
if $n_1=n_2=2$, then
$\scalar[\Sigma]=0$.
\end{theorem}

\medbreak\noindent{\bf Note:} In Example~\ref{Ex8.3}, we will present an example where $n_1=n_2=2$
and where $|K|[\Sigma]\ne0$ so this result is non-trivial.

Although Theorem~\ref{T1.8} shows $\scalar[\Sigma]$ is well defined, it does not provide
a useful upper bound for the total Gauss curvature of $\Sigma$. Suppose the roots of $\mathcal{P}_1$
and $\mathcal{P}_2$ are real, simple, and dominant. This means that
\begin{equation}\label{E1.e}
\begin{array}{l}
\sigma_1(t_1)=(e^{r_1t_1},\dots,e^{r_kt_1})\text{ for }r_1>0>r_k,\\
\sigma_2(t_2)=(e^{s_1t_2},\dots,e^{s_\ell t_2})\text{ for }s_1>0>s_\ell\,.
\end{array}\end{equation}
In this setting, we combine
Theorem~\ref{T1.4}, Equation~(\ref{E1.c}), and Theorem~\ref{T1.11} to obtain:

\begin{theorem}\label{T1.13}
If the roots of $\mathcal{P}_1$
and $\mathcal{P}_2$ are real and simple, and if the real roots are dominant, then
$
|\scalar|[\Sigma]\le 2\pi+4n_1(n_1-1)+4n_2(n_2-1)\,.
$
\end{theorem}

\begin{remark}\rm We will show in Example~\ref{Ex8.5} that this result is non-trivial;
$\scalar[\Sigma]$ does not vanish identically if $n_1>2$ and $n_2>2$. Furthermore, we will
give a family of surfaces $\Sigma_k$ in $\mathbb{R}^8$ where one of the sub-dominant roots
is complex where $\lim_{k\rightarrow\infty}\scalar[\Sigma_k]=-\infty$ so there is no
universal bound in this setting.
\end{remark}

Let $L_{ij}$ be the second fundamental form; this is vector valued and takes values in $T\Sigma^\perp$
(see Section~\ref{S2} for details). The {\it mean curvature} vector $H$ is given by:
$$H=g^{ij}L_{ij}\in T\Sigma^\perp\,.$$
The surface is minimal if and only if $H=0$. In Section~\ref{S7} we will show the surface is asymptotically
minimal if there are at least two positive and at least two negative roots for each ODE:
\begin{theorem}\label{T1.15}
Assume that the roots of $\mathcal{P}_1$
and $\mathcal{P}_2$ are real and simple, and that the real roots are dominant.
Assume that $r_1>r_2>0>r_{k-1}>r_k$ and that
$s_1>s_2>0>s_{\ell-1}>s_\ell$.
\begin{enumerate}
\item There exists $\epsilon=\epsilon(\Sigma)>0$ and $C=C(\Sigma)>0$ so $\|H\|\le Ce^{-\epsilon\|(t_1,t_2)\|}$.
\item $H\in L^3(\Sigma,\operatorname{dvol})$.
\end{enumerate}\end{theorem}
\begin{remark}\rm
In Example~\ref{Ex8.7}, we will show $H$ need not be bounded if $0>r_2$ and $0>s_2$.
Fix $p<3$. In Example~\ref{Ex8.8}, we will exhibit a surface $\Sigma_p$ satisfying the hypotheses of Theorem~\ref{T1.15}
where $H$ does not belong to $L^p$. This shows that $p=3$ is the best universal estimate.
\end{remark}

Throughout this paper, we will let $C=C(\Sigma)$ denote a generic positive constant that
can depend on $\Sigma$ but not on $(t_1,t_2)$.

\section{The geometry of surfaces embedded in $\mathbb{R}^n$}\label{S2}
Let $\Sigma(t_1,t_2)$ be an immersed surface in $\mathbb{R}^n$.
The components $g_{ij}$ of the Riemannian metric and the Riemannian measure
$\operatorname{dvol}$ on $\Sigma$
are defined by setting:
\begin{eqnarray*}
&&g_{ij}:=(\partial_{t_i}\Sigma,\partial_{t_j}\Sigma)\text{ and }
\operatorname{dvol}:=gdt_1dt_2\text{ where }
g:=\sqrt{g_{11}g_{22}-g_{12}g_{12}}\,.
\end{eqnarray*}
Let $\nabla$ be the Levi-Civita connection of $\Sigma$. If $\pi_\Sigma$
denotes orthogonal projection on the tangent space of $\Sigma$ and if $X$ and $Y$ are tangent vector fields
along $\Sigma$, then:
$$\nabla_XY=\pi_\Sigma\{XY(\Sigma)\}\,.$$
The curvature tensor is given by $R(X,Y)Z:=(\nabla_X\nabla_Y-\nabla_Y\nabla_X-\nabla_{[X,Y]})Z$.
The second fundamental form $L(X,Y)$ is defined to be:
\begin{equation}\label{E2.a}
L(X,Y)=(1-\pi_\Sigma)\{XY(\Sigma)\}\,.
\end{equation}
The second fundamental form is vector valued and takes values in $T\Sigma^\perp$.
Let $\{X,Y\}$ be linearly independent tangent vector fields along $\Sigma$.
The Gauss curvature $\scalar$ is given by:
$$\scalar:=(R(X,Y)Y,X)g^{-2}\,.$$
One has the Theorema Egregium of Gauss:
\begin{equation}\label{E2.b}
\scalar=\{(L(X,X),L(Y,Y))-(L(X,Y),L(X,Y))\}g^{-2}\,.
\end{equation}
If $\sigma$ is a curve in $\Sigma$ and if $\nu$ is a unit normal
to $\dot\sigma$ in $\Sigma$,  the geodesic curvature is:
$$\kappa_g(\sigma):=(\nabla_{\dot\sigma}\dot\sigma,\nu)\|\dot\sigma\|^{-2}\,.$$
This vanishes if and only if $\sigma$ is a geodesic and changes sign if we change the sign
of the normal.

We now introduce a convenient
formalism to discuss various geometric quantities in terms of wedge products.
Although the formulas are well-known, we shall give the proofs to establish notation.
Fix a point $(a,b)$ of $\Sigma$. Let $\gamma_1(t_1):=\Sigma(t_1,b)$ and $\gamma_2(t_2):=\Sigma(a,t_2)$ be
the coordinate curves through $(a,b)$. Let
$$
\Sigma_{/i}:=\partial_{t_i}\Sigma,\quad\Sigma_{/ij}:=\partial_{t_i}\partial_{t_j}\Sigma,\quad
\tilde L_{ij}:=
\Sigma_{/1}\wedge\Sigma_{/2}\wedge\Sigma_{/ij}\,.
$$
\begin{lemma}\label{L2.1}
Let $\{e_1,e_2\}$ be an orthonormal frame for $T\Sigma$ so $\Sigma_{/1}\wedge\Sigma_{/2}$ is
a positive multiple of $e_1\wedge e_2$.  Choose the normal to $\gamma_1$ in $\Sigma$
which points in the direction of $\Sigma_{/2}$.
\begin{enumerate}
\item $g=||\Sigma_{/1}\wedge\Sigma_{/2}||$.
\item $\kappa_g(\gamma_1)=(\Sigma_{/1}\wedge\Sigma_{/2},
\Sigma_{/1}\wedge\Sigma_{/11})\cdot g^{-1}\|\Sigma_{/1}\|^{-3}$.
\item $\Sigma_{/1}\wedge\Sigma_{/2}\wedge\Sigma_{/ij}=ge_1\wedge e_2\wedge L_{ij}$.
\item $\scalar=g^{-4}\{(\tilde L_{11},\tilde L_{22})-(\tilde L_{12},\tilde L_{12})\}$.\end{enumerate}
\end{lemma}

\begin{proof} Fix a point $P\in\Sigma$ and let $\{e_1,e_2\}$ be an orthonormal basis
for $T_P\Sigma$. Complete $\{e_1,e_2\}$ to an orthonormal basis $\{e_1,\dots,e_n\}$ for
$\mathbb{R}^n$. We may assume the basis chosen so that $\Sigma_{/1}=a_1e_1$ and $\Sigma_{/2}=b_1e_1+b_2e_2$
where $b_2>0$. Then
$$
\Sigma_{/1}\wedge\Sigma_{/2}=a_1b_2e_1\wedge e_2\quad\text{ and }\quad
||\Sigma_{/1}\wedge\Sigma_{/2}||^2=a_1^2b_2^2\,.
$$
We show that $g=\|\Sigma_{/1}\wedge\Sigma_{/2}\|$ and establish Assertion~1 by computing:
\begin{eqnarray*}
&&g_{11}=a_1^2,\quad g_{22}=b_1^2+b_2^2,\quad
     g_{12}=a_1b_1,\\
&&g^2=g_{11}g_{22}-g_{12}^2=a_1^2(b_1^2+b_2^2)-a_1^2b_1^2=a_1^2b_2^2\,.
\end{eqnarray*}
With our normalizations, $e_2$ is the normal to $\dot\gamma_1$
in $\Sigma$ which points in the direction of $\Sigma_{/2}$. Further normalize the orthonormal frame so that
$\Sigma_{/11}=c_1e_1+c_2e_2+c_3e_3$.
We prove Assertion~2 by computing:
\begin{eqnarray*}
&&\Sigma_{/1}\wedge\Sigma_{/2}=a_1b_2e_1\wedge e_2=ge_1\wedge e_2,\\
&&\Sigma_{/1}\wedge\Sigma_{/11}
=a_1c_2e_1\wedge e_2+a_1c_3e_1\wedge e_3,\\
&&(\Sigma_{/1}\wedge\Sigma_{/2},\Sigma_{/1}\wedge\Sigma_{/11})
=a_1c_2g,\\
&&\kappa_g(\gamma_1)=c_2a_1^{-2}
=(\Sigma_{/1}\wedge\Sigma_{/2},\Sigma_{/1}\wedge\Sigma_{/11})
a_1^{-3}g^{-1}\,.
\end{eqnarray*}
The second fundamental form $L_{ij}$ of
Equation~(\ref{E2.a}) is
the projection of $\Sigma_{/ij}$
on $T\Sigma^\perp$. Expand
$$\Sigma_{/ij}=\Gamma_{ij}{}^1e_1+\Gamma_{ij}{}^2e_2+L_{ij}{}^3e_3
+\dots+L_{ij}{}^ne_n
$$
where the $\Gamma_{ij}{}^k$ are the Christoffel symbols of the Levi-Civita connection and where
the second fundamental form
is given by
$L_{ij}=L_{ij}{}^3e_3+\dots+L_{ij}{}^ne_n$.
By Assertion~1, $\Sigma_{/1}\wedge\Sigma_{/2}=ge_1\wedge e_2$.
We derive Assertion~3 and Assertion~4 from Equation~(\ref{E2.b}) and complete the proof by computing:
$$\begin{array}{l}
\displaystyle\tilde L_{ij}=ge_1\wedge e_2\wedge\sum_{\nu=1}^nL_{ij}{}^\nu e_\nu=
g\sum_{\nu=3}^nL_{ij}{}^\nu e_1\wedge e_2\wedge e_\nu,\\
\displaystyle(\tilde L_{ij},\tilde L_{kl})=g^2\sum_{\nu=3}^nL_{ij}{}^\nu L_{kl}{}^\nu=g^2(L_{ij},L_{kl}),\\
\displaystyle\scalar=g^{-2}\{(L_{11},L_{22})-(L_{12},L_{12})\}
=g^{-4}\{(\tilde L_{11},\tilde L_{22})-(\tilde L_{12},\tilde L_{12})\}\,.
\end{array}$$
\end{proof}

\section{The proof of Theorem~\ref{T1.7}}\label{Sect-3}
Assume that all the roots of $\mathcal{P}_1$ and of $\mathcal{P}_2$ are simple and that the real roots are dominant.
We adopt the notation of Section~\ref{S1.5} throughout.  We shall concentrate on the first quadrant
$t_1\ge0$ and $t_2\ge0$ for the most part as the remaining quadrants can be handled similarly by reparametrizing
$\Sigma$ to set $\tilde t_i=\pm t_i$ as necessary. Set
\begin{equation}\label{E3.a}\epsilon_1=\epsilon_1(\Sigma)=\min\{r_1,s_1,-r_k,-s_\ell\}>0\,.
\end{equation}
Choose $\alpha_1\in\mathcal{R}_1-\{r_1\}$ so
$a_1:=\Re(\alpha_1)$ is maximal. Similarly, choose
$\beta_1\in\mathcal{R}_2-\{s_1\}$ so $c_1:=\Re(\beta_1)$ is maximal; both $\alpha_1$ and $\beta_1$
are sub-dominant.
Let
\begin{equation}\label{E3.b}
\mathcal{G}(t_1,t_2):=e^{2r_1t_1+(s_1+c_1)t_2}+e^{(r_1+a_1)t_1+2s_1t_2}\,.
\end{equation}
The following estimates are fundamental:

\begin{lemma}\label{L3.1}
Assume that all the roots of $\mathcal{P}_1$ and of $\mathcal{P}_2$ are simple and
that the real roots are dominant.
There exist $C_i=C_i(\Sigma)>0$ so that:
\begin{enumerate}
\item If $t_1\ge 0$ and $t_2\ge 0$, then $C_1\mathcal{G}(t_1,t_2)\le g(t_1,t_2)\le C_2\mathcal{G}(t_1,t_2)$.
\item For any $(t_1,t_2)\in\mathbb{R}^2$,
 $\|\Sigma(t_1,t_2)\|\ge\epsilon_1\|(t_1,t_2)\|$.
\end{enumerate}
\end{lemma}

\begin{proof}
Assertion~1 will show that $g$ and $\mathcal{G}$ grow at approximately the same rate on the first
quadrant. We begin the proof of Assertion~1 by estimating $g$ from below.
Suppose first that $\beta_1=c_1$ is real. We consider two of the coordinate functions which define $\Sigma$,
$\{\psi_1(t_1,t_2):=e^{r_1t_1}e^{s_1t_2},\ \psi_2(t_1,t_2):=e^{r_1t_1}e^{c_1t_2}\}$.
We use Lemma~\ref{L2.1} to estimate:
\begin{eqnarray*}
g&=&\|\Sigma_{/1}\wedge\Sigma_{/2}\|
\ge\left|\partial_{t_1}\psi_1\cdot\partial_{t_2}\psi_2-\partial_{t_1}\psi_2\cdot\partial_{t_2}\psi_1\right|\\
&=&r_1(s_1-c_1)e^{2r_1t_1+(s_1+c_1)t_2}\,.
\end{eqnarray*}
If, on the other hand, $\beta_1=c_1+d_1\sqrt{-1}$ for $d_1\ne0$,
then we consider the three coordinate functions:
$$\begin{array}{ll}
\psi_1(t_1,t_2):=e^{r_1t_1+s_1t_2},&\psi_2(t_1,t_2):=e^{r_1t_1+c_1t_2}\cos(d_1t_2),\\
\psi_3(t_1,t_2):=e^{r_1t_1+c_1t_2}\sin(d_1t_2)
\end{array}$$
and estimate similarly
\begin{eqnarray*}
g&\ge&\sum_{1\le i<j\le 3}
\left\{(\partial_{t_1}\psi_i\cdot\partial_{t_2}\psi_j-\partial_{t_1}\psi_j\cdot\partial_{t_2}\psi_i)^2\right\}^{1/2}\\
&\ge& r_1(s_1-c_1)e^{2r_1t_1+(s_1+c_1)t_2}\,.
\end{eqnarray*}
We have shown $g\ge Ce^{2r_1t_1+(s_1+c_1)t_2}$ for some $C$.
By reducing $C$ if necessary, we have similarly that $g\ge Ce^{(r_1+a_1)t_1+2s_1t_2}$. We average
these two estimates to establish the lower bound of Assertion~1 by showing:
$$
g\ge\textstyle\frac12C\mathcal{G}(t_1,t_2)\,.
$$

To establish the upper estimate of Assertion~1, we shall assume, for the sake of simplicity, that
all the roots are real as that is the case in which we shall use it; the general case can be dealt with using
the arguments above. The coordinate functions of $\Sigma$ take the form $\phi_{ij}(t_1,t_2)=e^{r_it_1+s_jt_2}$. Then
\begin{eqnarray*}
g^2&=&\|\Sigma_{/1}\wedge\Sigma_{/2}\|^2=
\textstyle\frac12\sum_{(i,j)\ne(a,b)}
\{\partial_{t_1}\phi_{ij}\partial_{t_2}\phi_{ab}-\partial_{t_1}\phi_{ab}\partial_{t_2}\phi_{ij}\}^2\\
&=&\textstyle\frac12\sum_{(i,j)\ne(a,b)}e^{2(r_i+r_a)t_1+2(s_j+s_b)t_2}(r_is_b-r_as_j)^2\,.
\end{eqnarray*}
If $i=a$, then $j\ne b$. Choose the notation so $1\le j<b$. We then have that
$2r_i+2r_a\le 4r_1$ and $2s_j+2s_b\le2s_1+2s_2$.
Thus we may bound
\begin{equation}\label{E3.c}
\{\partial_{t_1}\phi_{ij}\partial_{t_2}\phi_{ab}-\partial_{t_1}\phi_{ab}\partial_{t_2}\phi_{ij}\}^2
\le Ce^{4r_1t_1+(2s_1+2s_2)t_2}\,.
\end{equation}
On the other hand, if $i\ne a$,
choose the notation so that $1\le i<a$.
We then have $2r_i+2r_a\le 2r_1+2r_2$ and $2s_j+2s_b\le 4s_1$. The upper bound of Assertion~1
then follows Equation~(\ref{E3.c}) and from the estimate:
$$
\{\partial_{t_1}\phi_{ij}\partial_{t_2}\phi_{ab}-\partial_{t_1}\phi_{ab}\partial_{t_2}\phi_{ij}\}^2
\le Ce^{(2r_1+2r_2)t_1+4s_1t_2}\,.
$$

Suppose $t_1\ge0$ and $t_2\ge0$. Since $r_1>0$ and $s_1>0$, we may estimate
\begin{eqnarray*}
\|\Sigma(t_1,t_2)\|^2&\ge&e^{2r_1t_1+2s_1t_2}\ge\textstyle\frac12(2r_1t_1+2s_1t_2)^2\\
&\ge&2r_1^2t_1^2+2s_1^2t_2^2
\ge\epsilon_1^2\|(t_1,t_2)\|^2\,.
\end{eqnarray*}
Assertion~2 then follows for $t_1\ge0$ and $t_2\ge0$. We set $\tilde
t_i=\pm t_i$ as appropriate to reparametrize $\Sigma$ and
establish Assertion~2 in the remaining quadrants.\end{proof}

By Lemma~\ref{L3.1}, $g>0$. This implies $\Sigma$ is an immersion. We show that $\Sigma$ has infinite volume
by estimating
\begin{eqnarray*}
\operatorname{vol}(\Sigma)&=&\int_{-\infty}^\infty\int_{-\infty}^\infty g(t_1,t_2)dt_1dt_2\\
&\ge&
\int_{-\infty}^\infty\int_{-\infty}^\infty
r_1(s_1-c_1)e^{2r_1t_1+(s_1+c_1)t_2}dt_1dt_2=\infty\,.
\end{eqnarray*}
Let $C$ be a compact subset of $\mathbb{R}^n$. Since
$\Sigma$ is continuous and $C$ is closed, $\Sigma^{-1}(C)$ is closed. Since $C$ is compact, $C$ is bounded so we can find $R$
so that $\|C\|\le R$. Thus if $(t_1,t_2)\in\Sigma^{-1}(C)$, then $\epsilon_1\|(t_1,t_2)\|\le\|\Sigma(t_1,t_2)\|\le R$.
This shows that $\Sigma^{-1}(C)$ is bounded and hence, being closed, is compact. Since the inverse image of a compact
set is compact, $\Sigma$ is a proper map.

Let $\sigma(u)=\Sigma(t_1(u),t_2(u))$ be a unit speed geodesic in $\Sigma$. Then $\|\dot\sigma\|=1$
and $\ddot\sigma(u)\perp T_{\sigma(u)}\Sigma$.
Choose a maximal domain $[0,u_0)$ for $\sigma$. Suppose $u_0<\infty$.
As $\sigma$ is a unit speed curve in $\mathbb{R}^n$,
$$
\|\sigma(0)-\sigma(u)\|\le u_0\quad\text{so}\quad \|\sigma(u)\|\le\|\sigma(0)\|+u_0
$$
 for $u<u_0$. We use Lemma~\ref{L3.1} to see that
$$
\epsilon_1\|(t_1(u),t_2(u))\|\le\|\sigma(u)\|\le u_0+\|\sigma(0)\|\,.
$$
Since $(t_1(u),t_2(u))$ is uniformly bounded, we may
choose a sequence of values $u_n$ which converge to $u_0$
so that $\{t_1(u_n)\}$ and $\{t_2(u_n)\}$ are convergent sequences, i.e. so that for some $(t_1^0,t_2^0)$
we have that:
$$\lim_{n\rightarrow\infty}(t_1(u_n),t_2(u_n))=(t_1^0,t_2^0)\,.$$
Since $\Sigma$ is continuous, this implies $\lim_{n\rightarrow\infty}\sigma(u_n)$ exists and belongs to
$\Sigma$. This implies that $\sigma$ can be extended smoothly beyond the limiting value of $u_0$;
this contradiction shows $\Sigma$ is geodesically complete.

Let $\Psi_1(t_1,t_2):=e^{r_1t_1+s_1t_2}$ and $\Psi_2(t_1,t_2):=e^{r_kt_1+s_1t_2}$ be two
of the coordinate functions of $\Sigma$.
Suppose that $\Sigma(t_1,t_2)=\Sigma(\tilde t_1,\tilde t_2)$. Then $\Psi_1(t_1,t_2)=\Psi_1(\tilde t_1,\tilde t_2)$
and $\Psi_2(t_1,t_2)=\Psi_2(\tilde t_1,\tilde t_2)$. Consequently:
$$e^{(r_1-r_k)t_1}=\Psi_1(t_1,t_2)\Psi_2(t_1,t_2)^{-1}=\Psi_1(\tilde t_1,\tilde t_2)\Psi_2(\tilde t_1,\tilde t_2)^{-1}
=e^{(r_1-r_k)\tilde t_1}\,.$$
Since $r_1-r_k>0$, we conclude $t_1=\tilde t_1$. A similar argument shows $t_2=\tilde t_2$ so $\Sigma$ is 1-1.
This completes the proof of Theorem~\ref{T1.7}.\hfill\qed

\begin{remark}\rm It is possible to prove Theorem~\ref{T1.7} under somewhat weaker assumptions.
If we assume there exist roots $\lambda_1,\lambda_2\in\mathcal{R}_1$
and $\lambda_3,\lambda_4\in\mathcal{R}_2$
so that $\Re(\lambda_1)>0>\Re(\lambda_2)$ and $\Re(\lambda_3)>0>\Re(\lambda_4)$,
then Theorem~\ref{T1.7}
continues to hold. We omit details in the interests of brevity.
\end{remark}

\section{The proof of Theorem~\ref{T1.8}}\label{S4}
We now examine the Gauss curvature $\scalar$.
We suppose $t_1\ge0$ and $t_2\ge0$ as the remaing 3 quadrants can be
handled similarly. We begin with the following estimate:

\begin{lemma}\label{L4.1}
There exists a constant $C=C(\Sigma)$ so that if $t_1\ge0$ and $t_2\ge0$ then:
\begin{enumerate}
\item $(\tilde L_{11},\tilde L_{22})\le Ce^{(4r_1+2a_1)t_1+(4s_1+2c_1)t_2}$.
\item $(\tilde L_{12},\tilde L_{12})\le Ce^{(4r_1+2a_1)t_1+(4s_1+2c_1)t_2}$.
\end{enumerate}\end{lemma}

\begin{proof} Expand
$$
\sigma_1(t_1)=\sum_{i=0}^{n_1-1}\phi_i(t_1)e_i\text{ and }\sigma_2(t_2)=\sum_{j=0}^{n_2-1}\psi_j(t_2)f_j\,.
$$
We assume $\phi_0(t_1)=e^{r_1t_1}$ and $\psi_0(t_2)=e^{s_1t_2}$. We also assume that
$\phi_i(t_1)$ and $\psi_j(t_2)$ for $i\ge1$ and $j\ge1$ are of the form:
$$\phi_i(t_1)=e^{a_it_1}\cdot\left
\{\begin{array}{r}1\\\cos(b_it_1)\\\sin(b_it_1)\end{array}\right\}\text{ and }
\psi_i(t_2)=e^{c_jt_2}\cdot\left
\{\begin{array}{r}1\\\cos(d_jt_2)\\\sin(d_jt_2)\end{array}\right\}
$$
where $r_1=a_0>a_1\ge a_2\dots$ and
$s_1=c_0>c_1\ge c_2\ge\dots$ so the remaining exponentials grow less rapidly.
Exponential growth of the form
$$e^{(a_{i_1}+a_{i_2}+a_{i_3})t_1+(c_{j_1}+c_{j_2}+c_{j_3})t_2}$$
in $\Sigma_{/1}\wedge\Sigma_{/2}\wedge\Sigma_{/ij}$ arises from terms of the form
$$(e_{m_1}\otimes f_{k_1})\wedge(e_{m_2}\otimes f_{k_2})\wedge(e_{m_3}\otimes f_{k_3})$$
where
$\{i_1,i_2,i_3\}$ is a permutation of $\{m_1,m_2,m_3\}$ and
 $\{j_1,j_2,j_3\}$ is a permutation of $\{k_1,k_2,k_3\}$. Possible terms of maximal growth can
be enumerated as follows:
\smallbreak\noindent{\bf Type 1.} Terms involving $e^{3r_1t_1}$. This corresponds to $i_1=i_2=i_3=0$
and hence $\{k_1,k_2,k_3\}$ are distinct. Thus these grow at most like $e^{(s_1+2c_1)t_2}$ in $t_2$.
\smallbreak\noindent{\bf Type 2.} Terms involving $e^{3s_1t_2}$. This corresponds to
$j_1=j_2=j_3=0$ and hence $\{m_1,m_2,m_3\}$ are distinct. Thus these grow at most like
$e^{(r_1+2a_1)t_1}$ in $t_1$.
\smallbreak\noindent{\bf Type 3.} Terms involving at least 2 different exponentials in $t_1$ and at least
2 different exponentials in $t_2$. Here at least one of the $\{i_1,i_2,i_3\}$ involves an index which is not $0$
and at least one of the $\{j_1,j_2,j_3\}$ involves an index which is not $0$. Thus those grow at most like
$e^{(2r_1+a_1)t_1+(2s_1+c_1)t_2}$.
\medbreak When considering $(\tilde L_{ij},\tilde L_{kl})$, terms must be paired against like terms. Let
$$\xi_{M,K}:=(e_{m_1}\otimes f_{k_1})\wedge(e_{m_2}\otimes f_{k_2})\wedge(e_{m_3}\otimes f_{k_3})\,.$$
Then $(\xi_{M,K},\xi_{\tilde M,\tilde K})=0$ if
$\{m_1,m_2,m_3\}$ is not a permutation of $\{\tilde m_1,\tilde m_2,\tilde m_3\}$ or if
$\{k_1,k_2,k_3\}$ is not a permutation of $\{\tilde k_1,\tilde k_2,\tilde k_3\}$.
Thus
terms of Type 1 must be paired against terms of Type 1, of Type 2 against Type 2, and of Type 3 against Type 3.
We consider $\tilde L_{11}=\Sigma_{/1}\wedge\Sigma_{/2}\wedge\Sigma_{/11}$.
We have
$$\Sigma_{/1}\wedge\Sigma_{/11}=
((\partial_{t_1}\sigma_1\otimes\sigma_2)\wedge(\partial_{t_1}\partial_{t_1}\sigma_1\otimes\sigma_2))\,.$$
In this expression, $\sigma_2$ can be treated as a constant vector and essentially ignored for the moment.
Since at least 2 different terms must occur in any non-zero wedge product, there are no $e^{2r_1t_1}$ exponentials
appearing. Thus there are no terms of Type 1 in $\tilde L_{11}$. Similarly there are no terms of Type 2 in
$\tilde L_{22}$. Thus $(\tilde L_{11},\tilde L_{22})$ contains only terms of Type 3 so Assertion~1 follows.
Next, we shall consider $\tilde L_{12}=\Sigma_{/1}\wedge\Sigma_{/2}\wedge\Sigma_{/12}$. We have
$$\Sigma_{/1}\wedge\Sigma_{/12}=(\partial_{t_1}\sigma_1\otimes\sigma_2)
\wedge(\partial_{t_1}\sigma_1\otimes\partial_{t_2}\sigma_2)\,.$$
In this expression, $\partial_{t_1}\sigma_1$ can be treated as a constant vector and essentially ignored for
the moment. Since at least 2 different terms must occur in any non-zero term, there are no $e^{2s_1t_2}$
exponentials appearing. Thus there are no terms of Type 2 to be considered and, similarly no terms of Type 1 to
be considered and Assertion~2 follows.
\end{proof}

We apply Lemma~\ref{L2.1} and Lemma~\ref{L4.1} to estimate therefore that:
\begin{equation}\label{E4.a}
\begin{array}{l}
|\scalar|\le Cg^{-4}e^{(4r_1+2a_1)t_1+(4s_1+2c_1)t_2},\\
g|\scalar|\le Cg^{-3}e^{(4r_1+2a_1)t_1+(4s_1+2c_1)t_2}\,.
\end{array}\end{equation}
We use Lemma~\ref{L3.1} to estimate
$g^2\ge\epsilon^2e^{(3r_1+a_1)t_1+(3s_1+c_1)t_2}$. Raising this to the third and fourth power yields
\begin{equation}\label{E4.b}
\begin{array}{l}
g^4\ge\epsilon^4e^{(6r_1+2a_1)t_1+(6s_1+2c_1)t_2},\\
g^3\ge\epsilon^3e^{(\frac92r_1+\frac32a_1)t_1+(\frac92s_1+\frac32c_1)t_2}\,.
\end{array}\end{equation}
Let $\epsilon_1$ be as in Equation~(\ref{E3.a}). Choose
$\epsilon_2=\epsilon_2(\Sigma)>0$ to measure the spectral gap, i.e. so:
\begin{eqnarray*}
r_1-\epsilon_2\ge\Re(\lambda)\ge r_k+\epsilon_2\text{ for all }\lambda\in\mathcal{R}_1-\{r_1,r_k\},\\
s_1-\epsilon_2\ge\Re(\mu)\ge s_\ell+\epsilon_2\text{ for all }\mu\in\mathcal{R}_2-\{s_1,s_\ell\}\,.
\end{eqnarray*}
Combining Equation~(\ref{E4.a}) with Equation~(\ref{E4.b}) then yields the estimates:
\begin{eqnarray*}
|\scalar|&\le&Ce^{(-2r_1t_1-2s_1t_2)}\le Ce^{-2\epsilon_1(t_1+t_2)}\le Ce^{-2\epsilon_1\|(t_1,t_2)\|},\\
g|\scalar|&\le&Ce^{((4-\frac92)r_1+(2-\frac32)a_1)t_1+((4-\frac92)s_1+(2-\frac32)c_1)t_2}\\
&=&Ce^{-\frac12(r_1-a_1)t_1-\frac12(s_1-c_1)t_2}\le Ce^{-\frac12(\epsilon_2t_1+\epsilon_2t_2)}
\le Ce^{-\frac12\epsilon_2\|(t_1,t_2)\|}\,.
\end{eqnarray*}
This establishes Assertion~1 and Assertion~2 on the first quadrant $t_1\ge0$ and $t_2\ge0$; we use
similar arguments to establish these estimates in the remaining quadrants.
Integrating the estimate for $g|\scalar|$ in polar coordinates
then shows $|\scalar|[\Sigma]\le C\epsilon_2^{-1}$ which completes the proof of Theorem~\ref{T1.8}.\hfill\qed

\begin{remark}\rm
 It is not necessary to assume that roots $\mu$ of $\mathcal{P}_1$
 with $r_1>\Re(\mu)>r_k$ are simple;
multiple roots can appear in this range as the exponential estimates swamp any
powers of $t_1$. Similarly, it is not
necessary to assume that the remaining roots $\mu$ of $\mathcal{P}_2$ with
 $s_1>\Re(\mu)>s_\ell$ are simple; the arguments
go through unchanged. More care must be taken, however, if the dominant roots $r_1$ or $r_k$ of
$\mathcal{P}_1$ or the
dominant roots $s_1$ or $s_\ell$ of $\mathcal{P}_2$ are not simple and
a further investigation by the second author into this case is planned.
\end{remark}

\section{The proof of Theorem~\ref{T1.11}}\label{S5}

Adopt the notation of Definition~\ref{D1.10}.

\begin{lemma}\label{L5.1} Let $\gamma_{\pm r}(t):=\Sigma(t,\pm r)=\sigma_1(t)\otimes\sigma_2(\pm r)$.
If all the roots of $\mathcal{P}_1$ and of $\mathcal{P}_2$ are simple and if the real roots are dominant, then:
\begin{eqnarray*}
\lim_{r\rightarrow\infty}\int_{-r}^r\kappa_g(\gamma_{\pm r})(t)ds=-\Theta[\sigma_1]\,.
\end{eqnarray*}
\end{lemma}

\begin{proof} We will use the inward unit normal to apply
the Gauss-Bonnet theorem. This points in the direction of $\mp\Sigma_{/1}\wedge\Sigma_{/2}(t,\pm r)$.
Lemma~\ref{L2.1} to shows:
$$
\kappa_g(t,\pm r)ds=\mp(\Sigma_{/1}\wedge\Sigma_{/2},\Sigma_{/1}\wedge\Sigma_{/11})
\cdot g^{-1}\|\Sigma_{/1}\|^{-2}(t,\pm r)dt\,.
$$
First let $t_2=r$.
We express
$\sigma_2(t_2)=e^{s_1t_2}(f_1+\mathcal{E}(t_2))$ where the
remainder $\mathcal{E}(t_2)$ is exponentially suppressed, i.e. satisfies
an estimate of the form $\|\mathcal{E}(t_2)\|\le e^{-\epsilon t_2}$ for some $\epsilon>0$ if $t_2>>0$.
In this setting, to simplify the notation, we
shall simply write $\sigma_2(t_2)\sim e^{s_1t_2}f_1$. We compute:
\begin{eqnarray*}
&&\begin{array}{ll}\Sigma_{/1}\sim\dot\sigma_1\otimes e^{s_1r}f_1,&
\Sigma_{/2}\sim\sigma_1\otimes s_1e^{s_1r}f_1,\\
g=\|\Sigma_{/1}\wedge\Sigma_{/2}\|\sim|s_1|e^{2s_1r}\|\dot\sigma_1\wedge\sigma_1\|,&
\Sigma_{/11}\sim\ddot\sigma_1\otimes e^{s_1r}f_1,\end{array}\\
&&\kappa_g(\gamma_r)ds=-(\Sigma_{/1}\wedge\Sigma_{/2},\Sigma_{/1}\wedge\Sigma_{/11})g^{-1}\|\Sigma_{/1}\|^{-2}dt\\
&&\qquad\qquad\sim-\frac{s_1e^{4s_1r}}{|s_1|e^{4s_1r}}
\frac{(\dot\sigma_1(t_1)\wedge\sigma_1(t_1),\dot\sigma_1(t_1)
\wedge\ddot\sigma_1(t_1))}{\|\sigma_1(t_1)\wedge\dot\sigma_1(t_1)\|\cdot
\|\dot\sigma_1(t_1)\|^2}dt
\end{eqnarray*}
This gives $-\Theta(\sigma_1)dt$ in the limit since $s_1>0$. We do not need to change the sign of the normal
but again get a negative sign if $s_k<0$ since $+\frac{s_k}{|s_k|}=-1$.
  \end{proof}

  We apply the Gauss-Bonnet theorem to the square $\Sigma([-r,r]\times[-r,r])$. Let $\alpha_i$ be
  the interior angles. We then have:
   \begin{eqnarray*}
 2\pi&=&\int_{-r}^r\int_{-r}^r\scalar(t_1,t_2)gdt_1dt_2+\sum_{i=1}^4(\pi-\alpha_i)
+\int_{-r}^r\kappa_g(\Sigma(t,r))ds\\&&+ \int_{-r}^r\kappa_g(\Sigma(t,-r))ds
+ \int_{-r}^r\kappa_g(\Sigma(r,t))ds+ \int_{-r}^r\kappa_g(\Sigma(-r,t))ds\,.
 \end{eqnarray*}
 We examine the angle $\alpha_1$ at $\Sigma(r,r)$. Because
 $\Sigma_{/1}(r,r)\sim r_1\Sigma(r,r)$ and because
 $\Sigma_{/2}(r,r)\sim s_1\Sigma(r,r)$, $\Sigma_{/1}$ and $\Sigma_{/2}$ point
 in approximately the same direction. Consequently, $\cos(\alpha_1)\sim 1$
and $\alpha_1\sim 0$. Keeping careful track of the signs shows the other angles also are close to $0$.
  Theorem~\ref{T1.11} then follows from Lemma~\ref{L5.1}.\hfill\qed

\section{The proof of Theorem~\ref{T1.12}}\label{S6}
 We apply Theorem~\ref{T1.11} to the setting
$n_1=n_2$. Let $\{\xi_1,\xi_2\}$ be the standard orthonormal basis for $\mathbb{R}^2$.
Suppose $\sigma(t)=e^{at}e_1+e^{bt}e_2$ for $a>0>b$. We use Equation~(\ref{E1.d}) to see
that:
$$\Theta[\sigma]=\int_{-\infty}^\infty\frac{|(a-b)ab|
e^{(a+b)t}}{a^2e^{2at}+b^2e^{2bt}}dt
=\int_{-\infty}^\infty\frac{|(a-b)ab|e^{(a-b)t}}{a^2e^{2(a-b)t}+b^2}dt\,.
$$
We have $a-b>0$. We change variables setting $x:=e^{(a-b)t}$ to express
$$
\Theta[\sigma]=\int_0^\infty \frac{|ab|}{a^2x^2+b^2}dx
=\int_0^\infty\frac{|a|}{|b|}\frac1{\frac{a^2}{b^2}x^2+1}dx\,.
$$
We again change variables setting $y=\frac{|a|}{|b|}x$ to express
$$
\Theta[\sigma]=\int_0^\infty\frac1{y^2+1}dy=\frac\pi2\,.
$$
\medbreak\noindent Theorem~\ref{T1.12} now follows from Theorem~\ref{T1.11}.\hfill\qed

\section{The proof of Theorem~\ref{T1.15}}\label{S7}
Let $\Sigma(t_1,t_2)=\sigma_1(t_1)\otimes\sigma_2(t_2)$
where
$$
\begin{array}{l}
\sigma_1(t_1)=(e^{r_1t_1},\dots,e^{r_kt_1})\text{ for }r_1>r_2>0>r_{k-1}>r_k,\\
\sigma_2(t_2)=(e^{s_1t_2},\dots,e^{s_\ell t_2})\text{ for }s_1>s_2>0>s_{\ell-1}>s_\ell\,.
\end{array}$$
We focus on the first quadrant and assume $t_1\ge0$ and $t_2\ge0$; the other quadrants are handled
similarly. By Lemma~\ref{L3.1}, $g$ is growing exponentially at $\infty$ and the
growth rate is controlled by the function $\mathcal{G}$ of Equation~(\ref{E3.b}); this need not be
the case if $s_2<<0$ and $t_2<<0$. Let $\{e_i\}$ (resp. $\{f_a\}$ and $\{e_i\otimes f_a\}$) be an orthonormal basis
for $\mathbb{R}^{n_1}$ (resp. $\mathbb{R}^{n_2}$ and $\mathbb{R}^{n_1n_2}$) so that summing over
$i$, $a$, and $(i,a)$, yields:
$$
\sigma_{P_1}(t_1)=e^{r_it_1}e_i,\quad\sigma_{P_2}(t_2)=e^{s_at_2}f_a,\quad
\Sigma(t_1,t_2)=e^{r_it_1+s_at_2}e_i\otimes f_a\,.
$$
We express $\|H\|$ in terms of wedge products and establish its asymptotic growth rate at
infinity as follows:

\begin{lemma}\label{L7.1}
Let $\Sigma$ satisfy the hypotheses of Theorem~\ref{T1.15}. Adopt the notation established above.
\begin{enumerate}
\item Let $\mathfrak{H}:=\Sigma_{/1}\wedge\Sigma_{/2}\wedge(g_{11}\Sigma_{/22}+g_{22}\Sigma_{/11}
-2g_{12}\Sigma_{/12})\in\Lambda^{3}(\mathbb{R}^{n_1n_2})$. Then:
$$\|H\|=g^{-3}\|\mathfrak{H}\|\,.$$
\item Let $\mathcal{H}:=e^{5r_1t_1+(3s_1+s_2+s_3)t_2}+e^{(3r_1+r_2+r_3)t_1+5s_1t_2}
+e^{(4r_1+r_2)t_1+(4s_1+s_2)t_2}$. There exist constants $C_i=C_i(\Sigma)>0$ so that
if $t_1\ge0$ and if $t_2\ge0$,
then
$$C_1g^{-3}\mathcal{H}\le\|H\|\le C_2g^{-3}\mathcal{H}\,.$$
\end{enumerate}\end{lemma}

\begin{proof}
The mean curvature is given by
$H=g^{ij}L_{ij}\in T_P\Sigma^\perp$.
Let $\{\xi_1,\xi_2\}$ be an orthonormal frame for $T\Sigma$
so $\Sigma_{/1}\wedge\Sigma_{/2}=g\xi_1\wedge\xi_2$. By Lemma~\ref{L2.1},
$$\Sigma_{/1}\wedge\Sigma_{/2}\wedge\Sigma_{/ij}=g\xi_1\wedge\xi_2\wedge L_{ij}\,.$$
Since $g^{11}=g^{-2}g_{22}$, $g^{22}=g^{-2}g_{11}$, $g_{12}=-g^{-2}g_{12}$, and
since $\{\xi_1,\xi_2,L_{ij}\}$ form an orthogonal set, we prove
Assertion~1 by computing:
\begin{eqnarray*}
\|H\|&=&g^{-2}\|g_{22}L_{11}+g_{11}L_{22}-2g_{12}L_{12}\|\\
&=&g^{-3}\|\Sigma_{/1}\wedge\Sigma_{/2}\wedge\{g_{22}\Sigma_{/11}+g_{11}\Sigma_{/22}
-2g_{12}\Sigma_{/12}\}\|\\
&=&g^{-3}\|\mathfrak{H}\|\,.
\end{eqnarray*}

If $\{ua,vb,wc\}$ are distinct pairs of indices, set
$$\xi_{ua,vb,wc}:=(e_u\otimes f_a)\wedge(e_v\otimes f_b)\wedge(e_w\otimes f_c)\,.$$
If $\omega\in\Lambda^3(\mathbb{R}^{n_1n_2})$, let
$c(\xi_{ua,vb,wc},\omega)$ denote the coefficient of $\xi_{ua,vb,wc}$ in
$\omega$. Since $\omega=\sum_\xi c(\xi,\omega)\omega$, there exist constants $C_i=C_i(n_1,n_2)$ so that
$$C_1\sum_\xi |c(\xi,\omega)|\le\|\omega\|\le C_2\sum_\xi |c(\xi,\omega)|\,.$$
We wish to show that $\mathcal{H}$ controls the growth rate of $\|H\|$ at infinity. Thus we must
estimate each coefficient $c(\xi,\mathfrak{H})$ from above by $\mathcal{H}$ and exhibit 3 different
$\xi$ which we will use to estimate $\|\mathfrak{H}\|$ from below in terms of the 3 terms comprising $\mathcal{H}$.
We shall use the same argument given to establish Lemma~\ref{L4.1}. We may express
$$
c(\xi_{ua,vb,wc},\Sigma_{/1}\wedge\Sigma_2\wedge\Sigma_{/\mu\nu})(t_1,t_2)=
e^{(r_u+r_v+r_w)t_1+(s_a+s_b+s_c)t_2}c_{\mu\nu,ua,vb,wc}
$$
where
\medbreak\qquad
$c_{11,ua,vb,wc}=\det\left(\begin{array}{lll}
r_u&r_v&r_w\\
s_a&s_b&s_c\\
r_u^2&r_v^2&r_w^2\end{array}\right),\quad
c_{22,ua,vb,wc}=\det\left(\begin{array}{lll}
r_u&r_v&r_w\\
s_a&s_b&s_c\\
s_a^2&s_b^2&s_c^2\end{array}\right)$,
\medbreak\qquad
$c_{12,ua,vb,wc}=\det\left(\begin{array}{lll}
r_u&r_v&r_w\\
s_a&s_b&s_c\\
r_us_a&r_vs_b&r_ws_c\end{array}\right)$.
\smallbreak\noindent{\bf Terms of Type 1.} Suppose $u=v=w=1$. Then
$c(\xi_{1a,1b,1c},\tilde L_{11})=0$ and $c(\xi_{1a,1b,1c},\tilde L_{12})=0$.
Since $\{a,b,c\}$ are distinct, we may bound
\begin{equation}\label{E7.a}
|c(\xi_{1a,1b,1c},\mathfrak{H})(t_1,t_2)|\le|c_{22,1a,1b,1c}|e^{5r_1t_1+(3s_1+s_2+s_3)t_2}\le
C\mathcal{H}(t_1,t_2)\,.
\end{equation}
Let $\xi=\xi_{11,12,13}$. As $g_{22}\ge Ce^{2r_1t_1+2s_1t_2}$ and as
$c_{\mu\nu,1a,1b,1c}=0$ for $(\mu,\nu)\ne(2,2)$,
\begin{equation}\label{E7.b}
\begin{array}{l}
\|H(t_1,t_2)\|\ge|c_{22,11,12,13}|e^{2r_1t_1+2s_1t_2}e^{3r_1t_1+(s_1+s_2+s_3)t_2}\\
\quad=r_1(s_1 - s_2) (s_1 - s_3) (s_2 - s_3)e^{5r_1t_1+(3s_1+s_2+s_3)t_2}\,.\vphantom{\vrule height 14pt}
\end{array}\end{equation}

\smallbreak\noindent{\bf Terms of Type 2.} Suppose $a=b=c=1$. We argue similarly to conclude:
\begin{equation}\label{E7.c}\begin{array}{ll}
|c(\xi_{u1,v1,w1},\mathfrak{H})(t_1,t_2)|\le C\mathcal{H}(t_1,t_2),\\
\|H(t_1,t_2)\|\ge s_1(r_1 - r_2) (r_1 - r_3) (r_2 - r_3)e^{(3r_1+r_2+r_3)t_1+5s_1t_2}\,.\vphantom{\vrule height 12pt}
\end{array}\end{equation}

\smallbreak\noindent{\bf Terms of Type 3.} We suppose $(u,v,w)\ne(1,1,1)$ and $(a,b,c)\ne(1,1,1)$.
The following upper bound is then immeduate:
\begin{equation}\label{E7.d}
|c(\xi_{ua,vb,wc},\mathfrak{H})(t_1,t_2)|\le Ce^{(4r_1+r_2)t_1+(4s_1+s_2)t_2}\,.
\end{equation}
Let $\xi=\xi_{11,12,21}$. We expand, modulo lower order terms,
\begin{eqnarray*}
&&g_{11}=r_1^2e^{2r_1t_1+2s_1t_2}+\dots,\quad
g_{12}=r_1s_1e^{2r_1t_1+2s_1t_2}+\dots,\\
&&g_{22}=s_1^2e^{2r_1t_1+2s_1t_2}+\dots\,.
\end{eqnarray*}
We compute, again modulo lower order terms, that:
\begin{eqnarray*}
&&|c(\xi_{11,12,21},\mathfrak{H})(t_1,t_2)|
=\left| c(\xi_{11,12,21},g_{22}\tilde L_{11}+g_{11}\tilde L_{22}-2g_{12}\tilde L_{12})\right|\\
&=&e^{(4r_1+r_2)t_1+(4s_1+s_2)t_2}\left|
s_1^2c_{11,11,12,21}
+r_1^2c_{22,11,12,21}-2r_1s_1c_{12,11,12,21}\right|+\dots\\
&=&e^{(4r_1+r_2)t_1+(4s_1+s_2)t_2}
\{r_1 (r_1 - r_2) s_1 (s_1 - s_2) (2 r_1 s_1 - r_2 s_1 - r_1 s_2)\}+\dots\,.
\end{eqnarray*}
Since $\{r_1 (r_1 - r_2) s_1 (s_1 - s_2) (2 r_1 s_1 - r_2 s_1 - r_1 s_2)\}>0$, we have
\begin{equation}\label{E7.e}
|c(\xi_{11,12,21},\mathfrak{H})(t_1,t_2)|\ge Ce^{(4r_1+r_2)t_1+(4s_1+s_2)t_2}\,.
\end{equation}
Assertion~2 now follows from Equation~(\ref{E7.a})--Equation~(\ref{E7.e}).
\end{proof}

We restrict to the first quadrant $t_1\ge0$ and $t_2\ge0$.
By Lemma~\ref{L3.1},
\begin{equation}\label{E7.f}
g\ge C e^{2r_1t_1+(s_1+s_2)t_2}\quad\text{and}\quad g\ge Ce^{(r_1+r_2)t_1+2s_1t_2}
\end{equation}
for some $C>0$.  We use Equation~(\ref{E7.f}) to see if $\delta\in[0,1]$, then:
\begin{equation}\label{E7.g}
g^k\ge C e^{k\delta\{2r_1t_1+(s_1+s_2)t_2\}+k(1-\delta)\{(r_1+r_2)t_1+2s_1t_2\}}\,.
\end{equation}
We apply Lemma~\ref{L7.1}.
\subsection{The proof that $\|H\|$ is exponentially decaying}
 Let $\delta\in[0,1]$. We use Equation~(\ref{E7.g}) and Lemma~\ref{L7.1}. We bound
terms of Type I by:
\begin{eqnarray*}
&&g^{-3}e^{5r_1t_1+(3s_1+2s_2)t_2}\le Ce^{a_1(\delta)t_1+a_2(\delta)t_2}\quad\text{for}\\
&&a_1(\delta)=5r_1-3\{\delta2r_1+(1-\delta)(r_1+r_2)\}\text{ and }\\
&&a_2(\delta)=3s_1+2s_2-3\{\delta(s_1+s_2)+(1-\delta)2s_1\}\,.
\end{eqnarray*}
We show such terms exhibit exponential decay by estimating:
\begin{eqnarray*}
&&a_1(\textstyle\frac23)=5r_1-4r_1-r_1-r_2=-r_2<0,\\
&&a_2(\textstyle\frac23)=3s_1+2s_2-2s_1-2s_2-2s_1=-s_1<0\,.
\end{eqnarray*}
The terms of Type 2 are estimated similarly. We estimate the terms of Type 3:
\begin{eqnarray*}
&&g^{-3}e^{(4r_1+r_2)t_1+(4s_1+s_2)t_2}\le Ce^{a_1(\delta)t_1+a_2(\delta)t_2}\quad\text{for}\\
&&a_1(\delta)=4r_1+r_2-3\{\delta(2r_1)-(1-\delta)(r_1+r_2)\}\text{ and }\\
&&a_2(\delta)=4s_1+s_2-3\{\delta(s_1+s_2)-3(1-\delta)(2s_1)\}\,.
\end{eqnarray*}
We take $\delta=\frac12$ and show such terms exponential decay
by computing:
\begin{eqnarray*}
a_1(\textstyle\frac12)=4r_1+r_2-3r_1-\frac32r_1-\frac32r_2=-\frac12r_1-\frac12r_2<0,\\
a_2(\textstyle\frac12)=4s_1+s_2-\frac32s_1-\frac32s_2-3s_1=-\frac12s_1-\frac12s_2<0\,.
\end{eqnarray*}
This completes the proof that $\|H\|$ decays exponentially.\hfill\qed

\subsection{The proof that $\|H\|\in L^3(gdt_1dt_2)$.} We examine
$$
g\|H\|^3=g^{-8}\|\Sigma_{/1}\wedge\Sigma_{/2}\wedge\{g_{22}\Sigma_{/11}
+g_{11}\Sigma_{/22}-2g_{12}\Sigma_{/12}\}\|^3\,.
$$
We estimate the terms of type 1. Set
\begin{eqnarray*}
&&a_1(\delta):=3(5r_1)-8\{\delta 2r_1+(1-\delta)(r_1+r_2)\},\\
&&a_2(\delta):=3(3s_1+2s_2)-8\{\delta(s_1+s_2)+(1-\delta)2s_1\}\,.
\end{eqnarray*}
We take $\delta=\frac78$ to compute:
\begin{eqnarray*}
&&\textstyle a_1(\frac78)=15r_1-14r_1-r_1-r_2=-r_2<0,\\
&&\textstyle a_2(\frac78)=9s_1+6s_2-7s_1-7s_2-2s_1=-s_2<0\,.
\end{eqnarray*}
The terms of Type 2 are estimated similarly. To estimate the terms of Type 3, we take $\delta=\frac12$ and compute:
\begin{eqnarray*}
a_1(\textstyle\frac12)&=&3\{4r_1+r_2\}-8\{\delta(2r_1)+(1-\delta)(r_1+r_2)\}\\
&=&\textstyle12r_1+3r_2-8\{\frac32r_1+\frac12r_2\}=-r_2<0,\\
a_2(\textstyle\frac12)&=&3\{4s_1+s_2\}-8\{\delta(2s_1)+(1-\delta)(s_1+s_2)\}\\
&=&\textstyle12s_1+3s_2-8\{\frac32s_1+\frac12s_2\}=-s_2<0\,.
\end{eqnarray*}
This estimates all the terms comprising $g\|H\|^3$; thus $g\|H\|^3dt_1dt_2$ is integrable.\hfill\qed

\section{Examples}\label{S8}
In this section, we present a number of examples to illustrate various points; many
of them were Mathematica assisted and used a program developed by M. Brozos-Vazquez \cite{M2}.

\subsection{Finite total first curvature}
Theorem~\ref{T1.2} shows the total first curvature of $\sigma$
 is finite if all the roots of $\mathcal{P}$ are simple
and if the real roots of $\mathcal{P}$ are dominant. This can fail if a dominant root is complex.
\begin{example}\label{Ex8.1}
\rm Let $\sigma(t):=(e^t\cos(t),e^t\sin(t),e^{-t})$. The dominant root here is complex. We
show that $\kappa ds$ is not in $L^1$ by computing:
\begin{eqnarray*}
&&\dot\sigma=(e^t(\cos(t)-\sin(t)),e^t(\cos(t)+\sin(t)),-e^{-t}),\\
&&\ddot\sigma=(-2e^t\sin(t),2e^t\cos(t),e^{-t}),\\
&&\|\dot\sigma\wedge\ddot\sigma\|=\{4e^{4t}+10\}^{1/2},\\
&&\kappa ds=\{4e^{4t}+10\}^{1/2}\{2e^{2t}+e^{-t}\}^{-1}dt\,.
\end{eqnarray*}
\end{example}

\subsection{Finite total Gauss curvature}
Theorem~\ref{T1.8} shows that if all the roots of $\mathcal{P}_1$ and $\mathcal{P}_2$ are simple
and if the real roots are dominant, then the total Gauss curvature is finite. This can fail
if one of the dominant roots is complex.

\begin{example}\label{Ex8.2}\rm
Let $\sigma_1(t_1)=(e^{t_1}\cos(t_1),e^{t_1}\sin(t_1),e^{-t_1})$ and $\sigma_2(t_2)=(e^{t_2},e^{-t_2})$. Set
$\mathcal{E}_1:
=\left(e^{4 t_1-4 t_2}+e^{4 (t_1+t_2)}+4 e^{-4 t_1}+6 e^{4 t_1}+5 e^{-4 t_2}+5 e^{4 t_2}+2\right)$.
We use a Mathematica notebook \cite{M2} to see that :
$$g\scalar=-\frac{16 \left(e^{4 t_1}+2\right) \left(e^{4 t_2}+1\right) e^{10 t_1+6 t_2}
\mathcal{E}_1^{0.5}}{\left(2 e^{4 (t_1+t_2)}+e^{8 (t_1+t_2)}+6 e^{8 t_1+4 t_2}+5 e^{4 t_1+8 t_2}
+5 e^{4 t_1}+e^{8 t_1}+4 e^{4 t_2}\right)^2}\,.
$$
This permits to estimate for $t_1\ge0$ and $t_2\ge0$ that:
$$
g\scalar\le-\frac{16e^{4t_1+4t_2+10t_1+16t_2+(4t_1+4t_2)/2}}
{\left((2+1+6+5+5+1+4)e^{8t_1+8t_2}\right)^2}=-\frac{16}{24}e^{0t_1+6t_2}\,.
$$
Thus $g\scalar dt_1dt_2$ is not integrable for $0\le t_1<\infty$ and $0\le t_2<\infty$.
\end{example}

\subsection{The total Gauss curvature if $n_1=n_2=2$}
In Theorem~\ref{T1.13}, we showed that if the
roots of $\mathcal{P}_1$ and $\mathcal{P}_2$ are real and simple and if $n_1=n_2=2$,
then $\scalar[\Sigma]=0$. This result is non-trivial; there are examples where $|\scalar|[\Sigma]\ne0$
in this setting.

\begin{example}\label{Ex8.3} \rm
Let $\sigma_1(t_1)=(e^{t_1},e^{-2t_1})$
and $\sigma_2(t_2)=(e^{t_2},e^{-2t_2})$ for $a>0$ and $b>0$.
We use Mathematica \cite{M2} to compute:
\begin{eqnarray*}
&&g\scalar=
\frac{9. e^{-6 (t_1+t_2)} \left(e^{6 (t_1+t_2)}-4\right)}{\left(e^{-8 (t_1+t_2)} \left(e^{6 (t_1+t_2)}+e^{6 (2 t_1+t_2)}+e^{6 (t_1+2 t_2)}+4 e^{6 t_1}+4 e^{6 t_2}\right)\right)^{1.5}},\\
&&\int_{\mathbb{R}^2}g\scalar dt_1dt_2=0,\qquad\text{and}\qquad
\int_{\mathbb{R}^2}|g\scalar|dt_1 dt_2\approx.811319.
\end{eqnarray*}
The Gauss creature changes sign; it is positive for $6t_1+6t_2>\ln(4)$ and negative for $6t_1+6t_2<\ln(4)$.
It does not vanish identically and Theorem~\ref{T1.12} is non-trivial.\end{example}

\subsection{Uniform estimates on the first curvature}
Let $\sigma$ be defined by an ODE where the dominant roots of $\mathcal{P}$ are real.
If all the roots are real, then Theorem~\ref{T1.4} gives a uniform estimate for the total first curvature
which depends only on the dimension. If sub-dominant
complex roots are permitted, then no such uniform upper bound exists.
\begin{example}\label{Ex8.4}\rm
Let $\sigma_k(t)=(e^t,\cos(kt),\sin(kt),e^{-t})$ for $k\ge1$. We have
\begin{eqnarray*}
&&\dot\sigma_k(t)=(e^t,-k\sin(kt),k\cos(kt),-e^{-t}),\\
&&\ddot\sigma_k(t)=(e^t,-k^2\cos(kt),-k^2\sin(kt),e^{-t}),\\
&&\|\dot\sigma_k(t)\|^2=e^{2t}+k^2+e^{-2t},\quad
\|\dot\sigma_k(t)\wedge\ddot\sigma_k(t)\|\ge k^3,\\
&&\lim_{k\rightarrow\infty}\kappa[\sigma_k]\ge
\lim_{k\rightarrow\infty}\int_{-\infty}^\infty\frac{k^3}{e^{2t}+k^2+e^{-2t}}dt
\ge\lim_{k\rightarrow\infty}\int_0^1\frac{k^3}{e^2+k^2+1}dt\\
&&\qquad=\lim_{k\rightarrow\infty}\frac{k^3}{e^2+k^2+1}=\infty\,.
\end{eqnarray*}\end{example}

\subsection{Applying the Gauss-Bonnet Theorem with all real roots}
Because $\scalar[\Sigma]$ can be non-zero if $n_1>2$ and
$n_2>2$, Theorem~\ref{T1.13} is non-trivial.

\begin{example}\label{Ex8.5}\rm We examine the identity
$\mathcal{E}:=\scalar[\Sigma]-2\Theta[\sigma_1]-2\Theta[\sigma_2]+2\pi=0$.
We take $\sigma_1(t_1)=(e^{t_1},e^{a_1t_1},e^{a_2t_1})$ and $\sigma_2(t_2)=(e^{t_2},e^{b_1t_2},e^{b_2t_2})$
where $0\ge a_1>a_2$ and $0\ge b_1>b_2$. We computed \cite{M2} that:
$$\begin{array}{rrrrrrrrr}
a_1&a_2&b_1&b_2&\scalar[\sigma]&\Theta[\sigma_1]&\Theta[\sigma_2]&\mathcal{E}&|\scalar|[\Sigma]\\
0&-1&0&-1&-1.8649&1.10423&1.10423&1*10^{-3}&1.866\\
0&-2&0&-3&-2.0356&1.07859&1.04485&7*10^{-4}&2.26658\\
-1&-1.1&-1&-1.2&-1.51466&1.26238&1.09344&.06&1.73122\\
-1&-2&-1&-2&-1.96762&1.07875&1.07875&5*10^{-4}&2.27566\\
-1&-5&-1&-5&-1.96884&1.07859&1.07859&8*10^{-7}&2.3783\\
-2&-4&-1&-3&-1.88447&1.09513&1.10423&6*10^{-7}&2.56669\\
-5&-6&-1&-2&-2.17533&.975259&1.07861&1*10^{-4}&3.33547\\
-5&-6&-7&-8&-2.43838&.975259&.947119&5*10^{-5}&4.32915\\
\end{array}$$
These calculations show that $\scalar$ takes on both positive and negative values since
$|\scalar|[\Sigma]\ne \scalar[\Sigma]$. If, for example, $(a_1,a_2,b_1,b_2)=(-5,-6,-7,-8)$,
then
\begin{eqnarray*}
&&|\Theta|[\sigma_1]\approx2.03662\ne\Theta[\sigma_1]\approx.975259,\\
&&|\Theta|[\sigma_2]\approx2.10877\ne\Theta[\sigma_2]\approx.947119
\end{eqnarray*}
$\Theta$ takes on both positive and negative values.
\end{example}

\subsection{A uniform estimate on the Gauss curvature does not exist if complex roots are allowed}
If we allow complex roots, no uniform upper bound
is possible. We extend Example~\ref{Ex8.4} as follows.
\begin{example}\label{Ex8.6}
\rm Let $\Sigma_k(t_1,t_2):=(e^{t_1},\cos(kt_1),\sin(kt_1),e^{-t_1})\otimes(e^{t_2},e^{-t_2})$.
We use Mathematica \cite{M2} to express
$g\scalar:=-\mathcal{E}_1\cdot\mathcal{E}_2\cdot\mathcal{E}_3^{\frac12}\cdot\mathcal{E}_4^{-2}$ where:
\medbreak\qquad
$\mathcal{E}_1:=4 \left(e^{4 t_2}+1\right) e^{4 t_1+6 t_2}$,
\medbreak\qquad
$\mathcal{E}_2:=2 \left(k^2+1\right) k^2 e^{2 t_1}+2 \left(k^2+1\right) k^2 e^{6 t_1}+\left(k^2+1\right) e^{8 t_1}+k^2$
\medbreak\qquad\qquad\qquad
$+2 \left(k^4+3 k^2-1\right) e^{4 t_1}+1$,
\medbreak\qquad
$\mathcal{E}_3:=e^{-4 (t_1+t_2)} \left(2 k^2 e^{4 (t_1+t_2)}
+2 \left(k^2+1\right) e^{2 t_1+4 t_2}+2 \left(k^2+1\right) e^{6 t_1+4 t_2}\right.$
\medbreak\qquad\qquad\qquad
$\left.+\left(k^2+1\right) e^{2 t_1+8 t_2}+\left(k^2+1\right) e^{6 t_1+8 t_2}+\left(k^2+4\right) e^{4 t_1+8 t_2}\right.$
\medbreak\qquad\qquad\qquad
$\left.+\left(k^2+1\right) e^{2 t_1}+\left(k^2+1\right) e^{6 t_1}+\left(k^2+4\right) e^{4 t_1}+4 e^{8 t_1+4 t_2}+4 e^{4 t_2}\right)$
\medbreak\qquad
$\mathcal{E}_4:=2 k^2 e^{4 (t_1+t_2)}+2 \left(k^2+1\right) e^{2 t_1+4 t_2}+2 \left(k^2+1\right) e^{6 t_1+4 t_2}$
\medbreak\qquad\qquad\qquad
$+\left(k^2+1\right) e^{2 t_1+8 t_2}+\left(k^2+1\right) e^{6 t_1+8 t_2}+\left(k^2+4\right) e^{4 t_1+8 t_2}$
\medbreak\qquad\qquad\qquad
$+\left(k^2+1\right) e^{2 t_1}+\left(k^2+1\right) e^{6 t_1}+\left(k^2+4\right) e^{4 t_1}+4 e^{8 t_1+4 t_2}+4 e^{4 t_2}$.
\medbreak\noindent
As a function of $k$, this is behaving like $1\cdot k^4\cdot k\cdot k^{-4}$ and thus
the integral goes to infinity as $k\rightarrow\infty$.
We also examine $\Theta(\sigma_{1,k})$ computing
$$
\Theta(\sigma_{1,k})=\left\{\frac{4-k^4}{\left(k^2+e^{-2 t}+e^{2 t}\right)
 \left(k^2 e^{-2 t}+k^2 e^{2 t}+k^2+e^{-2 t}+e^{2 t}+4\right)^{0.5}}\right\}\,.
$$
This is growing linearly in $k$ as $k\rightarrow\infty$ and hence
$\lim_{k\rightarrow\infty}\Theta(\sigma_{1,k})=\infty$.
We examine the identity $\mathcal{E}_k:=\scalar[\Sigma_k]-2\Theta[\sigma_{1,k}]-2\Theta[\sigma_2]+2\pi=0$
numerically:
$$\begin{array}{rrrrrrr}
k&\scalar[\Sigma_k]&\Theta[\sigma_{1,k}]&\Theta[\sigma_2]&\mathcal{E}_k\\
0&-0.933127&1.10423&\frac\pi2&2.96624*10^{-9}\\
1&-2.15652&0.49253&\frac\pi2&-7.07655*10^{-10}\\
2&-4.74826&-0.803332&\frac\pi2&-4.54394*10^{-9}\\
3&-7.77242&-2.31541&\frac\pi2&6.01698*10^{-9}\\
4&-10.9544&-3.90643&\frac\pi2&-6.68724*10^{-8}\\
10&-30.8223&-13.8403&\frac\pi2&8.18756*10^{-7}\\
50&-165.483&-81.1709&\frac\pi2&1.25382*10^{-7}\\
200&-671.171&-334.015&\frac\pi2&.000010552\\
2000&-6739.86&-3368.36&\frac\pi2&-6.08458*10^{-6}\\
20000&-67426.9&-33711.9&\frac\pi2&-0.000982013\\
200000&-674297&-337147&\frac\pi2&-0.0211308
\end{array}$$
\end{example}

\subsection{The norm of the mean curvature vector}
Let
$$\sigma_1(t_1)=(e^{r_1t_1},\dots,e^{r_kt_k})\text{ and }\sigma_2(t_2)=(e^{s_1t_2},\dots,e^{s_\ell t_2})$$
for $r_1>\dots>r_k$ and $s_1>\dots>s_\ell$.
By Lemma~\ref{L3.1} and Lemma~\ref{L7.1}, there exist constants $C_i>0$ so
\smallbreak$\displaystyle
C_1\le\frac g{e^{2r_1t_1+(s_1+s_2)t_2}+e^{(r_1+r_2)t_1+2s_1t_2}}\le C_2$
\smallbreak$\displaystyle
\|H\|\ge C_3g^{-3}\{e^{5r_1t_1+(3s_1+s_2+s_3)t_2}+e^{(3r_1+r_2+r_3)t_1+5s_1t_2}
+e^{(4r_1+r_2)t_1+(4s_1+s_2)t_2}\}$.

\begin{example}\label{Ex8.7}
\rm
If we set $t_1=t_2=t$, $r_1=s_1$, $r_2=s_2$, and $r_3=s_3$, then we get
$$
\|H\|(t)\ge C\frac{e^{(8r_1+r_2+r_3)t}}{e^{3(3r_1+r_2)t}}=Ce^{(-r_1-2r_2+r_3)t}\,.$$
If we take $r_1=1$, $r_2=-3$, and $r_3=-4$, then $\|H\|\ge Ce^{(-1+6-4)t}$ and this tends to infinity as $t$ becomes
large. Thus Assertion~1 of Theorem~\ref{T1.15} can fail if we permit $r_2$ or $s_2$ to be negative.
\end{example}

\begin{example}\label{Ex8.8}
\rm We set $r_1=s_1$, $r_2=s_2$, and $r_3=s_3$. We restrict to $t_1\le t_2$ and estimate
\begin{eqnarray*}
&&\|H\|^2g\ge C\frac{e^{2\{(3r_1+r_2+r_3)t_1+5s_1t_2\}}}{e^{5((r_1+r_2)t_1+2r_1t_2)}}
=Ce^{(r_1-3r_2+2r_3)t_1},\\
&&\int_\Sigma\|H\|^2\operatorname{dvol}\ge\int_0^\infty\int_0^{t_2}Ce^{(r_1-3r_2+2r_3)t_1}dt_1dt_2\\
&&\qquad=\frac C{r_1-3r_2+2r_3}\int_0^\infty\{(e^{(r_1-3r_2+2r_3)t_2}-1\}dt_2\,.
\end{eqnarray*}
This is infinite provided $r_1-3r_2+2r_2>0$. We could, for example, take $r_1=10$, $r_2=2$, and $r_3=1$.
So in general $\|H\|$ is not in $L^2$. More generally, let $p>2$. We may estimate:
\begin{eqnarray*}
&&\|H\|^pg\ge C\frac{e^{p\{(3r_1+r_2+r_3)t_1+5r_1t_2\}}}{e^{(3p-1)\{(r_1+r_2)t_1+2r_1t_2\}}}\\
&&\int_\Sigma\|H\|^p\operatorname{dvol}\ge C\int_0^\infty\int_0^{t_2}e^{(r_1+(1-2p)r_2+pr_3)t_1+(2-p)r_1t_2}dt_1dt_2\\
&&\qquad=\frac C{(r_1+(1-2p)r_2+pr_3)}\int_0^\infty\{e^{((3-p)r_1+(1-2p)t_2+pr_3)t_2}-e^{(2-p)r_1t_2}\}dt_2
\end{eqnarray*}
This will be divergent if $r_1+(1-2p)r_2+pr_3>0$ and $(3-p)r_1+(1-2p)r_2+pr_3>0$. Given $2<p<3$, we can
take $r_1=1$ and $r_2$ and $r_3$ very close to zero to see these inequalities are satisfied and the integral is divergent.
Thus $p=3$ is the best that can be established in general although in specific cases, better convergence can be obtained.
\end{example}

\section*{Acknowledgements}
\noindent Research of P. Gilkey was partially supported by
MTM2013-41335-P with FEDER funds (Spain). Research of J. H. Park was
supported by Basic Science Research Program through the National
Research Foundation of Korea(NRF) funded by the Ministry of
Education.
 We are grateful for the helpful comments and
assistance of M. Brozos-V\'{a}zquez of the Universidade de Coru\~na
(Spain).

\end{document}